\documentclass[a4paper,11pt]{amsart}

\usepackage[left=1.2cm,right=1.2cm,top=3.5cm,bottom=3cm]{geometry}

\usepackage{amsmath,amssymb,amscd,amsfonts}
\usepackage[hyphenbreaks]{breakurl}
\usepackage{mathrsfs}
\usepackage{latexsym}
\usepackage[all]{xy}
\usepackage{verbatim}
\usepackage[usenames, dvipsnames]{color}
\usepackage{multirow}
\usepackage{url}
\usepackage{mathdots}
\usepackage{MnSymbol}
\usepackage{stmaryrd}
\allowdisplaybreaks

\usepackage{hyperref}
\hypersetup{colorlinks,breaklinks,
           linkcolor=MidnightBlue,urlcolor=MidnightBlue,
                       anchorcolor=MidnightBlue,citecolor=MidnightBlue}

\newcommand{\Z}{\mathbb{Z}}

\newcommand{\C}{\mathbb{C}}
\newcommand{\F}{\mathbb{F}}
\renewcommand{\P}{\mathbb{P}} %projective
\newcommand{\G}{\Gamma} %congruence subgroups
\newcommand{\T}{\mathbf{T}} %Hecke operator
\newcommand{\U}{\mathbf{U}}
\newcommand{\Tr}{\mathcal{T}} %Bruhat-Tits tree
\newcommand{\D}{\mathcal{D}}
\newcommand{\g}{\gamma}

\newcommand{\m}{\mathfrak{m}}

\renewcommand{\O}{\Omega}

\newcommand{\mf}{\,|_{k,m}}

\newtheorem{thm}{Theorem}[section]
\newtheorem{prop}[thm]{Proposition}

\newtheorem{defin}[thm]{Definition}
\newtheorem{rem}[thm]{Remark}

\renewcommand{\matrix}[4]{\left(\begin{array}{cc} {#1} & {#2} \\ {#3} & {#4} \end{array}\right)}

\title{On Drinfeld cusp forms of prime level}

\author{Andrea Bandini}

\address{{\sc Andrea Bandini}: Universit\`a di Pisa \\
Dipartimento di Matematica \\
Largo Bruno Pontecorvo 5\\
 56127 Pisa, Italy
   }

\email{andrea.bandini@unipi.it}

\author{Maria Valentino}

\address{{\sc Maria Valentino}: Universit\`a della Calabria\\
Dipartimento di Matematica e Informatica\\
Ponte P. Bucci, Cubo 30B
87036 Rende (CS), Italy}

\email{valentino@mat.unical.it}

\begin{document}

\begin{abstract}
Let $(P_d)$ be any prime of $\F_q[t]$ of degree $d$ and consider the space of Drinfeld cusp forms {\em of level
$P_d$}, i.e. for the modular group $\G_0(P_d)$. We provide a definition for oldforms and newforms of level $P_d$. 
Moreover, when the dimension of the vector space of oldforms is one 
and $P_1=t$ we prove that the space of cuspforms of level $t$ is the direct sum of oldforms and newforms and that 
the Hecke operator $\T_t$ acting on Drinfeld cusp forms of level 1 is injective, thus providing more evidence for the 
conjectures presented and stated in \cite{BV2} and \cite{BV3}. 
\end{abstract}

\maketitle

\section{Introduction}
Let $K$ be the global function field $\F_q(t)$, where $q$ is a power of a fixed prime $p\in\Z$,
 fix the prime $\frac{1}{t}$ at $\infty$ and denote by
$\mathcal{O}:=\F_q[t]$ its ring of integers (i.e., the ring of functions regular outside $\infty$).
Let $K_\infty=\F_q(\!(\frac{1}{t})\!)$ be the completion of $K$ at $\frac{1}{t}$ with ring of integers 
$\mathcal{O}_\infty=\F_q\llbracket \frac{1}{t}\rrbracket$
and denote by $\C_\infty$ the completion of an algebraic closure of $K_\infty$.\\
The {\em Drinfeld upper half-plane} is the set $\O:=\P^1(\C_\infty) - \P^1(K_\infty)$
together with a structure of rigid analytic space (see \cite{FvdP}).
The group $GL_2(K_\infty)$ acts on $\Omega$ via M\"obius transformation
\[  \left( \begin{array}{cc}
a & b  \\
c & d
\end{array} \right)(z)= \frac{az+b}{cz+d}. \]
Let $\G$ be an arithmetic subgroup of $GL_2(\mathcal{O})$, then $\G$ has finitely many cusps, i.e. equivalence classes for the 
action of $\G$ on $\P^1(K)$.
For $\g =\matrix{a}{b}{c}{d}\in GL_2(K_\infty)$, $k,m \in \Z$
and $\varphi:\O\to \C_\infty$, we define the $\mf\gamma$ operator by
\begin{equation}\label{Mod0}
(\varphi \mf \g)(z) := \varphi(\g z)(\det \g)^m(cz+d)^{-k}.
\end{equation}
Since for any $\gamma\in GL_2(\mathcal{O})$ one has $\det(\gamma)\in \F_q^*$, the integers $m$ can be cosidered modulo $q-1$.

\begin{defin}
A rigid analytic function $\varphi:\O\to \C_\infty$ is called a {\em Drinfeld modular function of weight $k$ and type $m$ for $\G$} if
\begin{equation}\label{Mod} (\varphi \mf \g )(z) =\varphi(z)\ \ \forall \g\in\G.  \end{equation}
A Drinfeld modular function $\varphi$ of weight $k\geqslant 0$ and type $m\in \Z/(q-1)\Z$ for $\G$ is called a {\em Drinfeld modular form} if
$\varphi$ is holomorphic at all cusps.\\
A Drinfeld modular form $\varphi$ is called a {\em cusp form} if  it vanishes at all cusps.\\
The space of Drinfeld modular forms
of weight $k$ and type $m$ for $\G$ will be denoted by $M_{k,m}(\G)$. The subspace of cuspidal modular forms is denoted by $S^1_{k,m}(\G)$.
\end{defin}
\noindent The above definition coincides with \cite[Definition 5.1]{B}, other authors require the function to be 
meromorphic (in the sense of rigid analysis, see for example \cite[Definition 1.4]{Cor}) and would call our 
functions {\em weakly modular}. 

\noindent We shall deal only with the arithmetic subgroups 
\[ \G=\G_0(\mathfrak{m}):=\left\{ \left( \begin{array}{cc}
a & b  \\
c & d
\end{array} \right) \in GL_2(\mathcal{O})\,:\, c\equiv 0\pmod{\mathfrak{m}} \right\}, \]
where $\mathfrak{m}$ is an ideal of $\mathcal{O}$, and we shall focus mainly on the cases $\mathfrak{m}=1$ (so that
$\G_0(1)=GL_2(\mathcal{O})\,$) and $\mathfrak{m}$ a prime ideal. When $\mathfrak{m}$ is prime we fix the monic irreducible 
generator $P_d$ of $\mathfrak{m}$ and will use simply $P_d$ or $(P_d)$ to denote the ideal.
The spaces $S^1_{k,m}(\G_0(\m))$ denote cusp forms of {\em level $\m$}.
We recall that spaces of Drinfeld modular forms of fixed weight and type are finite dimensional vector space over $\C_\infty$; for details on dimensions the reader is referred to \cite{G}. \medskip

Fix an ideal $\mathfrak{m}$ and a monic irreducible element $P_d$ of degree $d$ in $\mathcal{O}$. Assume $(P_d)$ does not divide $\m$
(which is the case we shall usually work with): we have the following Hecke operators acting, respectively,  on $S^1_{k,m}(\G_0(\m))$ and
 $S^1_{k,m}(\G_0(\m P_d))$:
\[ \T_{P_d}(\varphi)(z):=P_d^{k-m}(\varphi \mf \matrix{P_d}{0}{0}{1})(z)+ 
P_d^{k-m}\sum_{\begin{subarray}{c} Q\in \mathcal{O}\\ \deg Q <d  \end{subarray}} (\varphi \mf \matrix{1}{Q}{0}{P_d})(z) \]
and
\[ \U_{P_d}(\varphi)(z):=P_d^{k-m} \sum_{\begin{subarray}{c}Q\in \mathcal{O}\\ \deg{Q}<d  \end{subarray}} (\varphi\mf \matrix{1}{Q}{0}{P_d})(z).\]
We recall that the operator $\U_{P_d}$ is commonly called {\em Atkin-Lehner operator}, or simply Atkin-operator. \medskip

Using Teitelbaum's representation of cusp forms as cocycles (see \cite{Teit} or \cite{B}, a brief account of the formulas relevant for our computations is in \cite[Sections 2.3 and 2.4]{BV2}), in \cite{BV2} we were able to compute the 
matrix associated with the Atkin operator $\U_t$ acting on $S^1_{k,m}(\G_1(t))$ (where, as usual, 
\[ \G_1(t):=\left\{ \left( \begin{array}{cc}
a & b  \\
c & d
\end{array} \right) \in GL_2(\mathcal{O})\,:\, a,d\equiv 1\pmod{t}\ {\rm and}\ c\equiv 0\pmod{t} \right\}\ {\rm )} \]
and to isolate inside it the blocks referring to the action on the subspace $S^1_{k,m}(\G_0(t))$ (see \cite[Section 4]{BV2}).
In \cite{BV1} (for $\G_1(t)$) and \cite{BV3} (for $\G_0(t)$) we studied the properties of such matrix as a tool to
investigate the analogue of several classical (characteristic zero setting) issues related to Drinfeld cusp forms. In particular,
we considered problems like the structure of cusp forms of level $t$, the injectivity of $\T_t$, diagonalizability and slopes for $\U_t$, 
i.e. $t$-adic valuation of eigenvalues of $\U_t$.  Moreover, we collected data on the 
distribution of slopes (available on the web page  \href{https://sites.google.com/site/mariavalentino84/publications}{https://sites.google.com/site/mariavalentino84/publications}) as the weight varies, which led us to formulate various conjecture \`a la Gouv\^ea-Mazur (see \cite{GM1}) and on the existence
of families of Drinfeld cusp forms. For details see \cite[Section 5]{BV2} and \cite[Section 6]{BV3}. \\ 
We would like to mention that, building on such results, Hattori has recently proved a function field analogue of Gouv\^ea-Mazur's conjecture
(see \cite{Ha1}) and has made relevant progresses in the construction of ($p$-adic) families of Drinfeld modular forms (see \cite{Ha2}).
It is worth mentioning that, following a completely different (more geometric) approach, Nicole and Rosso in \cite{NR} have provided
deep results on the existence of families of modular forms in characteristic $p$. \medskip

In the present paper we shall address the following issues.
\begin{enumerate}
\item[\em i)] \underline{Structure of $S^1_{k,m}(\m P_d)$}. A major and basic topic in the study of classical modular forms is the splitting of $S_k(\G_0(N))$, for a general
level $N\in \Z$,  as oldforms, those coming from a lower level $M|N$, and newforms, i.e. the orthogonal 
complement of the space of oldforms with respect to the Petersson inner product (see \cite[Chapter 5]{DS}). 
In the positive characteristic setting we do not have an analogue of such product, therefore we need a different approach. 
In \cite[Section 3] {BV2} we defined oldforms and newforms of level $t$ and we also conjectured,
and proved in some particular cases, that $S^1_{k,m}(\G_0(t))$ is direct sum of newforms and oldforms.
Here we generalize all definitions to a general prime level $P_d$ and also prove some further results for the
case $P_1=t$. 
\item[\em ii)] {\underline{Injectivity of $\T_t$}. Building on the data mentioned above, we observed a 
 phenomenon that has no analogue in the characteristic 
zero setting, namely that the Hecke operator  $\T_t$ acting on  $S^1_{k,m}(GL_2(\mathcal{O}))$ seems to be injective,
and this would have consequences also on the diagonalizability of $\U_t$ acting on the space of oldforms
(see \cite[Section 3.2]{BV2}). 
In the paper \cite{BV3} we already gave evidence of this conjecture for some special cases, here we shall 
extend the cases in which we can prove the injectivity of $\T_t$. }
\end{enumerate} \medskip

The paper is organized as follows.

\noindent In Section \ref{SecNewOld} we supply definitions of oldforms and newforms. We consider the maps $\delta_1,\delta_{P_d}$, 
called {\em degeneracy maps}, from a lower level $S^1_{k,m}(\G_0(\mathfrak{m}))$ to an upper one $S^1_{k,m}(\G_0(\mathfrak{m}P_d))$ 
(Section \ref{SecDeg}) and use them to define oldforms. On the other side we have {\em trace maps} which go the other way around
and use them, together with the crucial ingredient of {\em Fricke involution}, to define newforms (Section \ref{SecTr}). 
Two main issues appear here:
\begin{itemize}
\item we define newforms only for prime level $P_d$ (hence for $\m=1$), the definition seems easily generalizable for traces but
we lack an involution {\em of level $\m$} to extend it in general;
\item as mentioned above, we do not have the analog of Petersson inner product in our setting, hence we need to prove that cusp forms
are direct sum of our oldforms and newforms to confirm that our definitions are the ``right'' ones. 
\end{itemize}
We use the interaction between degeneracy maps, trace maps and Hecke operators to provide a description of the kernels of
$\T_{P_d}$ and $\U_{P_d}$ (Propositions \ref{KerTPn} and \ref{PropKerUt}): in particular, the criterion 
\[ \varphi\in S^1_{k,m}(GL_2(\mathcal{O}))\ {\rm is\ in\ } \in Ker(\T_{P_d})\ {\rm if\ and\ only\ if\ }
\delta_1(\varphi)\in Ker(\U_{P_d}^2) \]
will be useful to prove the injectivity of $\T_t$ in the case presented in Section \ref{SecSpecCase}.
Moreover, in Theorem \ref{ThmCriterionDirSum}, we show an important criterion, which is a generalization 
of \cite[Theorem 5.1]{BV3}, to get the direct sum between oldforms and newforms by proving that it
is equivalent to the invertibility of the map $\mathcal{D}:=Id-P_d^{k-2m}(Tr')^2$.

\noindent In Section \ref{SecSpecCase} we specialize to the case $P_1=t$. Exploiting the linear algebra 
translation of our conjectures provided in \cite{BV3} and using the criterions above we shall prove the following.

\begin{thm} Assume that $\dim_{\C_\infty}S^1_{k,m}(GL_2(\mathcal{O}))=1$, then we have:
\begin{itemize}
\item {the operator $\T_t$ acting on $S^1_{k,m}(GL_2(\mathcal{O}))$ is injective (Theorem \ref{ThmInj});}
\item {the space $S^1_{k,m}(\G_0(t))$ is direct sum of newforms and oldforms (Theorem \ref{ThmDirSum}).}
\end{itemize}
\end{thm}

\section{Newforms and oldforms}\label{SecNewOld}
Here we define oldforms and newforms for a general prime level $P_d$; most of the formulas and definitions are straightforward and come from computations on Hecke operators and trace maps
(defined in \cite[Section 3]{Vi}) similar to the ones presented in \cite{BV2}, hence we often only provide the outcome and refer the reader to
those papers for the missing details.

\subsection{Degeneracy maps and oldforms}\label{SecDeg} Let $\mathfrak{m}$ be any ideal in $\mathcal{O}$ and consider the spaces of Drinfeld cusp forms 
$S^1_{k,m}(\G_0(\mathfrak{m}))$ and $S^1_{k,m}(\G_0(\mathfrak{m}P_d))$ of levels $\mathfrak{m}$ and $\mathfrak{m}P_d$ respectively.
We have two maps which produce {\em oldforms} in $S^1_{k,m}(\G_0(\mathfrak{m}P_d))$:
\begin{align*}
S^1_{k,m}(\mathfrak{m}) & \to S^1_{k,m}(\G_0(\mathfrak{m}P_d))\\
\delta_1\varphi& =\varphi \\
\delta_{P_d}\varphi&= (\varphi\mf \matrix{P_d}{0}{0}{1})(z)=P_d^m\varphi(P_d z)
\end{align*}

\begin{prop}\label{PropDeltaInj}
Assume that $(P_d)$ does not divide $\mathfrak{m}$, then the map 
\begin{align*}
(\delta_1,\delta_{P_d}): S^1_{k,m}(\G_0(\mathfrak{m}))^2 & \to S^1_{k,m}(\G_0(\mathfrak{m}P_d))\\
(\varphi_1 , \varphi_2) & \mapsto \delta_1\varphi_1 + \delta_{P_d}\varphi_2
\end{align*}
is injective.
\end{prop}

\begin{proof}
The proof works exactly as in \cite[Proposition 3.1]{BV2}, just replace the tree $\Tr_t$ used there with the Bruhat-Tits tree $\Tr_{P_d}$ at $P_d$ 
associated with $GL_2(K_{P_d})$ ($K_{P_d}$ being the completion of $K$ at the prime $P_d$).
\end{proof}

\begin{defin}
The space of oldforms of level $\mathfrak{m}$, denoted by $S^{1,old}_{k,m}(\G_0(\mathfrak{m}))$, is the subspace of 
$S^1_{k,m}(\G_0(\mathfrak{m}))$ generated by the set $\{ (\delta_1,\delta_{P_d})(\varphi_1,\varphi_2)\,:\,(\varphi_1,\varphi_2)\in 
S^1_{k,m}(\G_0(\mathfrak{m}/(P_d))^2\,,\ for\ all\ (P_d)|\mathfrak{m}\}$.
\end{defin}

Let $\varphi\in S^1_{k,m}(\G_0(\mathfrak{m}))$ and assume that $P_d$ does not divide $\mathfrak{m}$ so that we have ``different'' Hecke
 operators $\T_{P_d}$ and $\U_{P_d}$ on the levels $\mathfrak{m}$ and $\mathfrak{m}P_d$ respectively. Then the relations between the
 maps $\delta_1$ and $\delta_{P_d}$ and the Hecke operators are the following:
\begin{equation}\label{EqDeltaT}
\delta_1(\T_{P_d}\varphi)=P_d^{k-m}\delta_{P_d}\varphi+\U_{P_d}(\delta_1\varphi)
\end{equation}
\begin{align}\label{EqDeltaU}
\U_{P_d}(\delta_{P_d}\varphi)& =P_d^{k-m} \sum_{\begin{subarray}{c}Q\in \mathcal{O}\\ \deg{Q}<d  \end{subarray}} (\varphi \mf \matrix{P_d}{0}{0}{1} \matrix{1}{Q}{0}{P_d} )(z) \\ \nonumber
& = P_d^{k-m} \sum_{\begin{subarray}{c}Q\in \mathcal{O}\\ \deg{Q}<d  \end{subarray}} (\varphi \mf \matrix{1}{Q}{0}{1} \matrix{P_d}{0}{0}{P_d} )(z) \\ \nonumber
& = P_d^m \sum_{\begin{subarray}{c}Q\in \mathcal{O}\\ \deg{Q}<d  \end{subarray}}  \varphi(z) =0.
\end{align}

\begin{prop} \label{PropEigenvalues}
Assume that $(P_d)$ does not divide $\mathfrak{m}$, then
\[ \{\mathrm{Eigenvalues}\ \mathrm{of}\ {\U_{P_d}}_{|S^{1,old}_{k,m}(\G_0(\mathfrak{m}P_d))} \} = 
\{ \mathrm{Eigenvalues}\ \mathrm{of}\ \T_{P_d}\} \cup \{0\}.  \]
\end{prop}

\begin{proof}
Let $(\delta_1,\delta_{P_d})(\varphi,\psi)$ be an old eigenform for $\U_{P_d}$ of eigenvalue $\lambda$. Then
\begin{align*}
\lambda (\delta_1,\delta_{P_d})(\varphi,\psi) & = \U_{P_d}((\delta_1,\delta_{P_d})(\varphi,\psi))\\
& = \U_{P_d}(\delta_1\varphi)\\
& = \delta_1(\T_{P_d})-P_d^{k-m}\delta_{P_d}(\varphi)\\
& = (\delta_1,\delta_{P_d})(\T_{P_d}\varphi,-P_d^{k-m}\varphi)
\end{align*}
implies $\T_{P_d}\varphi=\lambda \varphi$ because of the injectivity of $(\delta_1,\delta_{P_d})$.

If $\T_{P_d}\varphi =\lambda \varphi$ with $\lambda\neq 0$ we have
\begin{align*}
\U_{P_d}((\delta_1,\delta_{P_d})(\varphi,-\frac{P_d^{k-m}}{\lambda}\varphi)) & = \U_{P_d}(\delta_1\varphi)\\
& = \delta_1(\T_{P_d}\varphi)-P_d^{k-m}\delta_{P_d}\varphi\\
& = \lambda \delta_1\varphi -P_d^{k-m}\delta_{P_d}\varphi\\
& = \lambda(\delta_1,\delta_{P_d})(\varphi,-\frac{P_d^{k-m}}{\lambda}\varphi).\qedhere
\end{align*}
\end{proof}

\noindent We have just seen that the behaviour of $\U_{P_d}$ on oldforms is analogous to the classical case: the eigenvalues for $\U_{P_d}$ 
verify equations like $X^2-\lambda X=0$ where $\lambda$ is a nonzero eigenvalue for $\T_{P_d}$ (in the classical case the equation was
$X^2-\lambda X+p^{k-1}=0$ which reduces to our one modulo $p$, see \cite[Section 4]{GM1}).

\begin{rem}\label{RemDiagOld}
Let $\varphi$ be an eigenvector for $\T_{P_d}$ of eigenvalue $\lambda$, then the matrix for the action of $\U_{P_d}$ on the couple 
$\{ \delta_1\varphi , \delta_{P_d}\varphi \}$ is $\left(\begin{array}{cc} \lambda & -P_d^{k-m} \\ 0 & 0 \end{array}\right)$. Hence
it is easy to see that, assuming $(P_d)$ does not divide $\m$, the operator $\U_{P_d}$ is diagonalizable on oldforms 
if and only if the operators
$\T_{P_d}$ are diagonalizable at lower levels and are injective. We believe $\U_{P_d}$ is diagonalizable in odd characteristic
(and, for $P_1=t$, we provided evidence for it in \cite{BV1} and \cite{BV3}) and this motivates our investigation on the injectivity
of the Hecke operators $\T_{P_d}$.
\end{rem}

The next proposition describes $Ker(\T_{P_d})$ and will be crucial in the proof of Theorem \ref{ThmInj}.

\begin{prop}\label{KerTPn}
Let $\varphi\in S^1_{k,m}(\G_0(\mathfrak{m}))$ such that $P_d\nmid \mathfrak{m}$, then $\varphi\in Ker(\T_{P_d})$ if and only if $\delta_1(\varphi)\in Ker(\U_{P_d}^2)$.
\end{prop}

\begin{proof}
By \eqref{EqDeltaT}, for any $\varphi\in Ker(\T_{P_d})$ one has $\U_{P_d}^2(\delta_1\varphi)=-P_d^{k-m}\U_{P_d}(\delta_{P_d}\varphi)=0$.\\
Now let $\varphi\in S^1_{k,m}(\G_0(\mathfrak{m}))$ be such that $\delta_1\varphi\in Ker(\U_{P_d}^2)$. Then
\begin{align*}  
0 & = \U_{P_d}^2(\delta_1\varphi) = \U_{P_d}(\delta_1(\T_{P_d}\varphi)-P_d^{k-m}\delta_{P_d}\varphi)\\
 & = \delta_1(\T_{P_d}^2\varphi)-P_d^{k-m}\delta_{P_d}(\T_{P_d}\varphi)\\
 & = \delta(\T_{P_d}^2\varphi, -P_d^{k-m}(\T_{P_d}\varphi)) .
\end{align*}
Since $\delta$ is injective we have $\T_{P_d}\varphi=0$.
\end{proof}

\subsection{Trace maps and newforms}\label{SecTr}
From now on we take $\mathfrak{m}=1$ and denote $(\delta_1,\delta_{P_d}):S^1_{k,m}(GL_2(\mathcal{O}))^2\rightarrow S^1_{k,m}(\G_0(P_d))$ 
simply by $\delta$. The reason for this is the crucial role played by the Fricke involution in the definition of the twisted trace and of 
newforms (see below): the trace map should be easily generalizable to any level $\mathfrak{m}$ just considering representatives for
$\G_0(\mathfrak{m}P_d)\backslash \G_0(\mathfrak{m})$ but we are still looking for the correct generalization of the Fricke involution.
We recall that a system of coset representative for $\G_0(P_d)\backslash GL_2(\mathcal{O})$ is 
\[  R=\left\{ Id, \matrix{0}{-1}{1}{Q}\ \mathrm{s.t.}\ Q\in \mathcal{O}\ \mathrm{and}\ \deg{Q}<d\right\}. \]
For details on some of the maps defined in this section see \cite{Vi}. 

\begin{defin}\label{DefFrTrTr'} We have the following maps defined on $S^1_{k,m}(\G_0(P_d))$:
\begin{itemize}
\item the {\em Fricke involution}, which preserves the space $S^1_{k,m}(\G_0(P_d))$, is represented by the matrix 
\[ \g_{P_d}:= \matrix{0}{-1}{P_d}{0}  \]
and defined by $\varphi^{Fr}=( \varphi\mf \g_{P_d})$;
\item the {\em trace map} is defined by
\begin{align*}  
Tr: S^1_{k,m}(\G_0(P_d)) & \to S^1_{k,m}(GL_2(\mathcal{O})) \\
 \varphi &\mapsto \sum_{\g\in R} (\varphi\mf \g)(z);
\end{align*}
\item the {\em twisted trace map} is defined by
\begin{align*}
Tr':S^1_{k,m}(\G_0(P_d)) & \to S^1_{k,m}(GL_2(\mathcal{O})) \\
\varphi &\mapsto Tr(\varphi^{Fr}).
\end{align*}
\end{itemize}
\end{defin}

We list here many useful formulas expressing the relations between these maps, the Hecke operators and the maps $\delta_1$ and $\delta_{P_d}$,
the proofs rely on matrix decomposition and on the definitions of the various maps and are similar to those in \cite[Section 3]{BV2}. Please note
that the first three formulas hold for cusp forms of level $P_d$, while the following ones hold for cusp forms of level 1. 

\noindent Let $\psi\in S^1_{k,m}(\G_0(P_d))$, then we have
\begin{align}\label{EqFrFr}
(\psi^{Fr})^{Fr} &= \left( (\psi \mf \matrix{0}{-1}{P_d}{0})(z)\right)^{Fr}\\ \nonumber
& = (\psi \mf \matrix{0}{-1}{P_d}{0}\matrix{0}{-1}{P_d}{0})(z)= P_d^{2m-k}\psi ;
\end{align}

\begin{align}\label{EqTr}
Tr(\psi) & = \psi + \sum_{\begin{subarray}{c} Q\in \mathcal{O}\\ \deg{Q}<d \end{subarray}} \left( \psi\mf \matrix{0}{-1}{P_d}{0}\matrix{1}{Q}{0}{P_d}\matrix{\frac{1}{P_d}}{0}{0}{\frac{1}{P_d}} \right) (z)\\ \nonumber 
& = \psi + P_d^{-m}\U_{P_d}(\psi^{Fr}) ;
\end{align}

\begin{align}\label{EqTr'}
Tr'(\psi)= \psi^{Fr}+P_d^{m-k}\U_{P_d}(\psi) .
\end{align}

\noindent Now let $\varphi\in S^1_{k,m}(GL_2(\mathcal{O}))$, then we have

\begin{align}\label{EqDelta1Fr}
(\delta_1\varphi)^{Fr} & = (\varphi\mf \matrix {0}{-1}{P_d}{0})(z)\\ \nonumber
& = (\varphi \mf \matrix{0}{-1}{1}{0}\matrix{P_d}{0}{0}{1})(z) =\delta_{P_d}\varphi ;
\end{align}

\begin{align}\label{EqDeltaFr}
(\delta_{P_d}\varphi)^{Fr} & =(\varphi \mf \matrix{P_d}{0}{0}{1}\matrix{0}{-1}{P_d}{0})(z) \\ \nonumber
 & = (\varphi\mf \matrix{0}{-1}{1}{0}\matrix{P_d}{0}{0}{P_d})(z) = P_d^{2m-k}\delta_1\varphi ;
\end{align}

\begin{align}
\U_{P_d}((\delta_1\varphi)^{Fr})=0 ;
\end{align}

\begin{align}
\U_{P_d}((\delta_{P_d}\varphi)^{Fr})=P_d^{2m-k}[\delta_1\T_{P_d}\varphi-P_d^{k-m}\delta_{P_d}\varphi] ;
\end{align}

\begin{align}\label{EqTrD1}
Tr(\delta_1\varphi)=\sum_{\g\in R} \varphi= \varphi ;
\end{align}

\begin{align}\label{EqTrD}
Tr(\delta_{P_d}\varphi) & =( \sum_{\g\in R} \varphi \mf \matrix{P_d}{0}{0}{1}\g )(z) \\ \nonumber 
& = (\delta_{P_d}\varphi )(z) + ( \sum_{\begin{subarray}{c} Q\in \mathcal{O} \\ \deg{Q}<d \end{subarray}} \matrix{0}{-P_d}{1}{Q}) (z)\\ \nonumber
& = P_d^{m-k}\T_{P_d}\varphi .
\end{align}

As an application we have an explicit description of the kernel of the Hecke operator $\U_{P_d}$.

\begin{prop}\label{PropKerUt}
We have $Ker(\U_{P_d})=Im(\delta_{P_d})$.
\end{prop}

\begin{proof} We have already seen that $Ker(\U_{P_d})\supseteq Im(\delta_{P_d})$. Now let $\varphi\in Ker(\U_{P_d})$
and note that, by \eqref{EqTr'}, $Tr'(\varphi)=\varphi^{Fr}\in S^1_{k,m}(GL_2(\mathcal{O}))$. Then it is easy to check that, with
$\psi:=P_d^{k-2m}\varphi^{Fr}\in S^1_{k,m}(GL_2(\mathcal{O}))$, one has $\delta_{P_d}(\psi) = \varphi$.
\end{proof}

\begin{defin}\label{DefNew}
The space of newforms of level $P_d$, denoted by $S^{1,new}_{k,m}(\G_0(P_d))$ is given by $Ker(Tr)\cap Ker(Tr')$.
\end{defin}

\begin{rem}\label{RemUStable}
From formulas \eqref{EqDeltaT} and \eqref{EqDeltaU}, it is easy to see that $\U_{P_d}$ preserves the space of oldforms (of any level).
For any newform $\varphi$ of level $P_d$ we have $Tr(\varphi)=Tr'(\varphi)=0$, hence \eqref{EqTr'} yields
$\U_{P_d}(\varphi)=-P_d^{k-m}\varphi^{Fr}$. Thus it immediately follows that $Tr(\U_{P_d}(\varphi))=Tr'(\U_{P_d}(\varphi))=0$, i.e.
$\U_{P_d}$ preserves newforms as well.
\end{rem}

\begin{rem}
The trace alone is not enough to isolate newforms: indeed let $\varphi\in S^1_{k,m}(GL_2(\mathcal{O}))$ be such that $\T_{P_d}\varphi=\lambda \varphi$
with $\lambda\neq 0$. Then one can check that
\[ \psi_1:=\delta_1\varphi-\frac{P_d^{k-m}}{\lambda}\delta_{P_d}\varphi\ \in Ker(Tr)\]
and
\[  \psi_2:=\frac{P_d^{k-m}}{\lambda}\delta_1\varphi-P_d^{k-2m}\delta_{P_d}\varphi\ \in Ker(Tr') \]
(recall that, by the proof of Proposition \ref{PropEigenvalues}, $\psi_1$ is an $\U_{P_d}$-eigenvector of eigenvalue $\lambda$).
In general, $\psi_1\notin Ker(Tr')$ and $\psi_2\notin Ker(Tr)$ unless $\lambda=\pm P_d^{k/2}$.
\end{rem}

The values $\pm P_d^{k/2}$ (i.e. the {\em slope} $\frac{k}{2}$ in the sense of \cite[Definition 3.4 and Remark 3.5]{BV2}) are the
only possible eigenvalues for newforms and we actually believe that they identify newforms, i.e. there are no oldforms with such
eigenvalues (this would have relevant consequences also on other conjectures like the one discussed in Section \ref{SecSpecCase},
see \cite[Remark 5.3]{BV3}).   
 
\begin{prop}\label{PropNewEigenvalues}
Let $\varphi\in S^1_{k,m}(\G_0(P_d))$ be a new $\U_{P_d}$-eigenform of eigenvalue $\lambda$, then $\lambda=\pm P_d^{k/2}$.
\end{prop}

\begin{proof}
By \eqref{EqTr} and \eqref{EqTr'}
\[ \varphi = -P_d^{-m} \U_{P_d}(\varphi^{Fr})\quad \mathrm{and}\quad \varphi^{Fr}=-P_d^{m-k}\U_{P_d}(\varphi).\]
It follows that
\begin{align*}
\lambda^2\varphi & = \lambda(\U_{P_d}\varphi) = \U_{P_d}^2\varphi \\
& = \U_{P_d}(-P_d^{k-m}\varphi^{Fr})\\
& = -P_d^{k-m} \U_{P_d}(\varphi^{Fr})= P_d^k\varphi .
\end{align*}
Hence $\lambda =\pm P_d^{k/2}$.
\end{proof}

The following important criterion is the analog of \cite[Theorem 5.1]{BV3}.

\begin{thm}\label{ThmCriterionDirSum} 
We have a direct sum decomposition
$S^1_{k,m}(\G_0(P_d))=S^{1,old}_{k,m}(\G_0(P_d)) \oplus S^{1,new}_{k,m}(\G_0(P_d))$ if and only if the map 
$\D:=Id-P_d^{k-2m}(Tr')^2$ is bijective.
\end{thm}

\begin{proof}
$(\Longleftarrow)$
We start by proving that the intersection between newforms and oldforms is trivial.\\
Let $\eta=\delta(\varphi,\psi)\in S^1_{k,m}(\G_0(P_d))$ be old and new. The following facts hold:
\begin{itemize}
\item $\eta=\varphi+\psi^{Fr}$ since $\varphi$ and $\psi$ are both of level 1;
\item $0=Tr(\eta)=Tr(\varphi)+Tr(\psi^{Fr})=\varphi +Tr'(\psi)$, so that $Tr'(\psi)=-\varphi$;
\item $0=Tr'(\eta)=Tr'(\varphi)+Tr'(\psi^{Fr})=0$.
\end{itemize}
From the last two equalities we get
\[  0=-Tr'(Tr'\psi)+Tr((\psi^{Fr})^{Fr})=-(Tr')^2(\psi)+P_d^{2m-k}Tr(\psi) . \]
So
\[ (Tr')^2\psi-P_d^{2m-k}\psi =0  \] 
and 
\[  (Id-P_d^{k-2m}(Tr')^2)\psi=\D\psi=0. \] 
Since, by hypothesis, $\D$ is invertible, this yields $\psi=0$ and $\varphi=-Tr'(\psi)=0$ as well.\\
Now we have to prove the sum condition. Given $\eta\in S^1_{k,m}(\G_0(P_d))$ it is sufficient to find
$\varphi_1,\varphi_2\in S^1_{k,m}(GL_2(\mathcal{O}))$ such that $\eta-\delta(\varphi_1,\varphi_2)$ is new, i.e.
we need to solve the following
\[
\left\{ \begin{array}{l}Tr(\eta-\delta(\varphi_1,\varphi_2))=0\\
Tr'(\eta-\delta(\varphi_1,\varphi_2)) =0  
\end{array}\right. \ . \]
These equations are equivalento to 
\begin{equation}\label{EqSystem1}
\left\{ \begin{array}{l} Tr(\eta)-\varphi_1-Tr(\delta_{P_d}\varphi_2)=0\\
Tr'(\eta)-Tr'(\varphi_1)-Tr'(\delta_{P_d}\varphi_2)=0
\end{array}\right., \quad {\rm i.e.}\quad    
\left\{\begin{array}{l}
Tr(\eta)-\varphi_1-Tr(\varphi_2^{Fr})=0\\
Tr'(\eta)-Tr'(\varphi_1)-Tr'(\varphi_2^{Fr})=0
\end{array}
\right. ,
\end{equation}
which finally leads to 
\begin{equation}\label{EqSystem2}
\left\{ 
\begin{array}{l}\varphi_1=Tr(\eta)-Tr(\varphi_2^{Fr})\\
Tr(\eta^{Fr})-Tr(\varphi_1^{Fr})-P_d^{2m-k}\varphi_2=0
\end{array}
\right. \ .
\end{equation}
Using the two equations of \eqref{EqSystem2} we have
\begin{align}\label{EqSystem3}
\varphi_2& =P_d^{k-2m}[Tr'(\eta)-Tr'(\varphi_1)] \\ \nonumber
&= P_d^{k-2m}[Tr'(\eta)-Tr'(Tr(\eta))+(Tr')^2(\varphi_2)]\,.
\end{align}
Then $\D\varphi_2=P_d^{k-2m}[Tr'(\eta)-Tr'(Tr(\eta))]$ and $\varphi_2=P_d^{k-2m}\D^{-1}(Tr'(\eta-Tr(\eta)))$.\\
Substituting the first expression for $\varphi_2$ found in \eqref{EqSystem3} in the first equation of \eqref{EqSystem2}, one has
\[
\varphi_1=Tr(\eta)-P_d^{k-2m}(Tr')^2\eta+P_d^{k-2m}(Tr')^2\varphi_1  ,\] 
which implies
\[  \D\varphi_1=Tr(\eta)-P_d^{k-2m}(Tr')^2\eta \]
and finally
\[  \varphi_1=\D^{-1}(Tr(\eta)-P_d^{k-2m}(Tr')^2\eta). \]

\noindent$(\Longrightarrow)$
Let $\eta\neq 0$ be such that $\eta\in Ker(\D)$. Then $P_d^{2m-k}\eta=(Tr')^2\eta$. Recall that $Tr^2=Tr$ (as for any trace map)
and apply $Tr$ to obtain
\begin{align*}
P_d^{2m-k}Tr(\eta) & = Tr(Tr'(Tr'\eta))\\
& = Tr(Tr((Tr'\eta)^{Fr})) = (Tr')^2(\eta).
\end{align*}
Therefore $Tr(\eta)=\eta$, so $\eta$ is old and it  is contained in the image of $\delta_1$. Observe that $\U_{P_d}(\eta)\neq 0$, otherwise,
by Proposition \ref{PropKerUt}, one would have $\eta\in Im(\delta_1)\cap Im(\delta_{P_d})=\{0\}$ (by Proposition \ref{PropDeltaInj}). 
In particular, by Remark \ref{RemUStable}, $\U_{P_d}(\eta)$ is old. Then
\begin{align*}
P_d^{2m-k}\eta & = (Tr')^2\eta\\
& = Tr'(Tr'(\eta))\qquad\qquad\qquad {\rm (apply\ \eqref{EqTr'})}\\
& = Tr'(\eta^{Fr}+P_d^{m-k}\U_{P_d}(\eta)) \\
& = Tr((\eta^{Fr})^{Fr})+P_d^{m-k}Tr'(\U_{P_d}(\eta))\\
& = P_d^{2m-k}Tr(\eta)+P_d^{m-k}Tr'(\U_{P_d}(\eta)) .
\end{align*} 
So, $Tr'(\U_{P_d}(\eta))=0$ (because $\eta$ is old with $Tr(\eta)=\eta$).\\
Finally note that, by equations \eqref{EqDeltaT}, \eqref{EqTrD1} and \eqref{EqTrD},
\begin{align*}
Tr(\U_{P_d}(\eta)) & = Tr(\delta_1\T_{P_d}(\eta))-P_d^{k-m}Tr(\delta_{P_d}\eta)\\
 & = \T_{P_d}(\eta) -P_d^{k-m}P_d^{m-k}\T_{P_d}(\eta)=0
\end{align*}
So, $\U_{P_d}(\eta)$ is also new and we do not have direct sum.
\end{proof}

From the above proof an easy calculation leads to
\[  Ker(\D)=\{ \delta_1\varphi : \varphi\in S^1_{k,m}(GL_2(\mathcal{O}))\ \mathrm{and}\ \T_{P_d}\varphi=\pm P_d^{k/2}\varphi  \}.\]
Indeed recall that for any cusp form $\psi$ of level 1 we have $\delta_{P_d}\psi= (\delta_1\psi)^{Fr}$, hence 
\begin{align*}
Tr'(Tr'(\delta_1\varphi)) & = Tr'(Tr(\delta_{P_d}\varphi)) \\
 & = P_d^{m-k}Tr'(\T_{P_d}\varphi) \\
 & = P_d^{m-k} Tr((\T_{P_d}\varphi)^{Fr}) \\
 & = P_d^{m-k} Tr(\delta_{P_d}\T_{P_d}\varphi)\\
 & = P_d^{2m-2k} \T_{P_d}^2\varphi
\end{align*}
  
\noindent Moreover, $\delta_1\varphi\in Ker(\D)$ implies:
\begin{itemize}
\item $\U_{P_d}(\delta_1\varphi)$ is old and new;
\item if $\T_{P_d}\varphi= P_d^{k/2}\varphi $ then $$\U_{P_d}(\delta_1\varphi-P_d^{k/2-m}\delta_{P_d}\varphi)= P_d^{k/2}(\delta_1\varphi-P_d^{k/2-m}\delta_{P_d}\varphi)$$ and $\delta_1\varphi-P_d^{k/2-m}\delta_{P_d}\varphi$ is old and new;%Moreover
%\[ (\delta_1\varphi-P_d^{k/2-m}\delta_{P_d}\varphi)^{Fr}=-P_d^{m-k/2}(\delta_1\varphi-P_d^{k/2-m}\delta_{P_d}\varphi)\,.  \]
\item if $\T_{P_d}\varphi= -P_d^{k/2}\varphi $ then $$\U_{P_d}(\delta_1\varphi+P_d^{k/2-m}\delta_{P_d}\varphi)=- P_d^{k/2}(\delta_1\varphi+P_d^{k/2-m}\delta_{P_d}\varphi)$$ and $\delta_1\varphi+P_d^{k/2-m}\delta_{P_d}\varphi$ is old and new. %As before
%\[ (\delta_1\varphi+P_d^{k/2-m}\delta_{P_d}\varphi)^{Fr}=P_d^{m-k/2}(\delta_1\varphi+P_d^{k/2-m}\delta_{P_d}\varphi)\,.  \]
\end{itemize}

\section{Special case: $P_1=t$.}\label{SecSpecCase}

For the level $P_1=t$ we explicitly computed the matrices associated to the operator $\U_t$, the Fricke involution and the trace maps
(see \cite[Section 4]{BV2} and \cite[Sections 3 and 4]{BV3}): for the convenience of the reader we are going to briefly describe here these 
matrices.\\
We recall that, in order to have $S^1_{k,m}(\G_0(t))\neq 0$, we need $k\equiv 2m\pmod{q-1}$.  
Moreover, it is always possible to find a $j\in \{0,1,\dots,q-2\}$ and a unique $n\in \Z_{\geqslant 0}$ such that
 $k=2(j+1)+(n-1)(q-1)$ ($j$ is related to the type $m$ by the relation $m\equiv j+1\pmod{q-1}$, see \cite[Section 4.3]{BV2}). 
 From now on, the letters $j$ and $n$ will always be linked to the weight $k$ by the previous formula, giving us information, respectively,  
on the type $m$ and the dimension of the matrix $U$ associated to $\U_t$ acting on $S^1_{k,m}(\G_0(t))$.\\
  We have
 \begin{equation}\label{EqAt}
U= M D:= M \left(\begin{array}{ccc}
t^{s_1} & \cdots & 0\\
 & \ddots & \\
0 & \cdots & t^{s_n}
\end{array}\right)
\end{equation}
where, for $1\leqslant i\leqslant n$, we put $s_i=j+1+(i-1)(q-1)$ (so that $s_i+s_{n+1-i}=k$ for $1\leqslant i\leqslant\frac{n}{2}$ 
or $1\leqslant i\leqslant\frac{n+1}{2}$ according to $n$ being even or odd) and, for even $n$, the matrix $M$ is

{ \small\[ M= \left(\begin{array}{cccccccc} m_{1,1} & m_{1,2} & \cdots & m_{1,\frac{n}{2}} & (-1)^{j+1}m_{1,\frac{n}{2}}
& \cdots & (-1)^{j+1}m_{1,2} & (-1)^{j+1}(m_{1,1}-1)\\
m_{2,1} & m_{2,2} & \cdots & m_{2,\frac{n}{2}} & (-1)^{j+1}m_{2,\frac{n}{2}} & \cdots & (-1)^{j+1}(m_{2,2}-1) & (-1)^{j+1}m_{2,1}\\
\vdots & \vdots &   & \vdots & \vdots &  & \vdots & \vdots\\
m_{\frac{n}{2},1} & m_{\frac{n}{2},2} & \cdots & m_{\frac{n}{2},\frac{n}{2}} & (-1)^{j+1}(m_{\frac{n}{2},\frac{n}{2}}-1) & \cdots &
(-1)^{j+1}m_{\frac{n}{2},2} & (-1)^{j+1}m_{\frac{n}{2},1}\\
m_{\frac{n}{2}+1,1} & m_{\frac{n}{2}+1,2} & \cdots & (-1)^j & 0 & \cdots & (-1)^{j+1}m_{\frac{n}{2}+1,2}
& (-1)^{j+1}m_{\frac{n}{2}+1,1}\\
\vdots & \vdots & \iddots  & \vdots & \vdots & \ddots & \vdots & \vdots\\
m_{n-1,1} & (-1)^j & \cdots  & 0 & 0 &  \cdots & 0 & (-1)^{j+1}m_{n-1,1}\\
(-1)^j & 0 & \cdots  & 0 & 0 & \cdots & 0 & 0
\end{array}
\right)\,,\] }
while for odd $n$ one just needs to modify the indices a bit and add the central $\frac{n+1}{2}$-th column
\[ (m_{1,\frac{n+1}{2}}, \cdots , m_{\frac{n-1}{2},\frac{n+1}{2}}, (-1)^j, 0, \cdots, 0) .\]
The entries of $M$ are the binomial coefficients in $\F_p$
\begin{equation}\label{spam}
\displaystyle{ m_{a,b}= \left\{\begin{array}{ll}
\displaystyle{-\left[\binom{j+(n-a)(q-1)}{j+(n-b)(q-1)} + (-1)^{j+1} \binom{j+(n-a)(q-1)}{j+(b-1)(q-1)}\right]} & {\rm if}\ a\neq b \\
\displaystyle{(-1)^j\binom{j+(n-a)(q-1)}{j+(a-1)(q-1)}} & {\rm if}\ a=b \end{array}\right. .}
\end{equation}
The other matrices associated to the relevant maps we used to define oldforms and newforms are the following:
\begin{itemize}
\item the matrix for the Fricke involution is 
\begin{equation}\label{MatrixFricke}
t^{m-k}F=t^{m-k}\left( \begin{array}{ccc} 0 & \dots  & (-t)^{s_n} \\ 
 & \iddots & \\ (-t)^{s_1} & \dots & 0  \end{array}\right) = t^{m-k} 
 \left( \begin{array}{ccc} 0 & \dots  & (-1)^{j+1}t^{s_n} \\ 
 & \iddots & \\ (-1)^{j+1}t^{s_1} & \dots & 0  \end{array}\right)\,.
\end{equation}
Note that, if we let $A$ be the antidiagonal matrix  
\begin{equation}\label{EqA} A= \left( \begin{array}{ccc} 0 & \dots  & (-1)^{j+1} \\ 
 & \iddots & \\ (-1)^{j+1} & \dots & 0  \end{array}\right),\end{equation}
we get $AF=D$;
\item from equation \eqref{EqTr} we find that the trace is represented by the matrix
\begin{equation}\label{eqTr} 
T:= I+t^{-m}MD(t^{m-k}F)=I+t^{-k}MAF^2 = I+MA 
\end{equation}
where $I$ is the identity matrix of dimension $n$;
\item the twisted trace is represented by
\begin{equation}\label{eqTr'}
  T'=t^{m-k}TF= t^{m-k}(F+ MD ).
\end{equation}  
\end{itemize}

\begin{rem}\label{RemMatrixT}
Note that $MA$ switches columns $i$ and $n+1-i$ in the matrix $M$ and multiplies everything by $(-1)^{j+1}$: looking at the 
description of $M$ we see that this produces a matrix which looks just like $M$ except for the fact that the $(-1)^j$ on the 
antidiagonal disappear and are substituted by $(-1)^j(-1)^{j+1}=-1$ on the diagonal. Therefore the matrix $T=I+MA$ is the
following (for even $n$)
{ \small\[ T= \left(\begin{array}{cccccccc} m_{1,1} & m_{1,2} & \cdots & m_{1,\frac{n}{2}} & (-1)^{j+1}m_{1,\frac{n}{2}}
& \cdots & (-1)^{j+1}m_{1,2} & (-1)^{j+1}m_{1,1}\\
m_{2,1} & m_{2,2} & \cdots & m_{2,\frac{n}{2}} & (-1)^{j+1}m_{2,\frac{n}{2}} & \cdots & (-1)^{j+1}m_{2,2} & (-1)^{j+1}m_{2,1}\\
\vdots & \vdots &   & \vdots & \vdots &  & \vdots & \vdots\\
m_{\frac{n}{2},1} & m_{\frac{n}{2},2} & \cdots & m_{\frac{n}{2},\frac{n}{2}} & (-1)^{j+1}m_{\frac{n}{2},\frac{n}{2}} & \cdots &
(-1)^{j+1}m_{\frac{n}{2},2} & (-1)^{j+1}m_{\frac{n}{2},1}\\
m_{\frac{n}{2}+1,1} & m_{\frac{n}{2}+1,2} & \cdots & 0 & 0 & \cdots & (-1)^{j+1}m_{\frac{n}{2}+1,2}
& (-1)^{j+1}m_{\frac{n}{2}+1,1}\\
\vdots & \vdots & \iddots  & \vdots & \vdots & \ddots & \vdots & \vdots\\
m_{n-1,1} & 0 & \cdots  & 0 & 0 &  \cdots & 0 & (-1)^{j+1}m_{n-1,1}\\
0 & 0 & \cdots  & 0 & 0 & \cdots & 0 & 0
\end{array}
\right)\,.\] }
As before, for odd $n$ one just needs to modify the indices a bit and add the central $\frac{n+1}{2}$-th column
\[ (m_{1,\frac{n+1}{2}}, \cdots , m_{\frac{n-1}{2},\frac{n+1}{2}}, 0, \cdots, 0) .\]
Hence $T$ is basically $M$ without the $(-1)^j$ on the antidiagonal and verifies a number of equations/relations like\begin{itemize}
\item $T=A+M$;
\item $T=TA$ (this comes directly from the previous one, to verify it via computations on the above matrix one has to note that for
odd $n$ and even $j$ the central column is identically 0 because of the formula \eqref{spam}, while for odd $j$ one is simply
multiplying the central column by 1);
\item $T^2=T$, like any trace map.
\end{itemize}
From these, one can produce various relations on $M$ (like $MAT=TM=0$ or, more surprisingly, $M^3=M$) with consequences, for example, on the diagonalizability of $M$,
but we shall not pursue this topic any further here.
\end{rem}

\noindent We also recall that $Im(\delta_1)=Ker(Tr-Id)$, i.e. in terms of matrices
\begin{equation}\label{EqMA}
Im(\delta_1)=Ker(MA).
\end{equation}

In \cite[Section 5]{BV2} we hinted at some conjectures which were stated more explicitly in \cite[Conjecture 1.1]{BV3}: among other things
we conjectured that for $P_1=t$  
\begin{enumerate}
\item $\T_t$ is injective;
\item $S^1_{k,m}(\G_0(t))$ is the direct sum of oldforms and newforms. 
\end{enumerate}
In \cite{BV3} we proved some special cases building on the analog of Theorem \ref{ThmCriterionDirSum} (one of the reasons which makes us believe
the conjectures should hold for any $P_d$) and on the above matrices/formulas (which are not avaliable for $d\geqslant 2$). In particular,
in \cite[Theorem 5.5]{BV3} we proved that when $\dim_{\C_\infty}(S^1_{k,m}(GL_2(\mathcal{O}))=0$ (i.e. there are no oldforms)
the matrix $M$ is antidiagonal and the conjectures hold: we shall now approach the case $\dim_{\C_\infty}(S^1_{k,m}(GL_2(\mathcal{O}))=1$,
this will include many more cases since, for example, $\dim_{\C_\infty}(S^1_{k,0}(GL_2(\mathcal{O}))=1$ if and only if
$q\leqslant n <2q-1$, by \cite[Proposition 4.3]{Cor} (compare with the bounds of \cite[Theorems 5.8, 5.9, 5.12, 5.14]{BV3}).

\subsection{Injectivity of $\T_t$}

\begin{thm}\label{ThmInj}
Assume that $\dim_{\C_\infty}Im (\delta_1)=1$, then $\T_t$ is injective.
\end{thm}

\begin{proof}
By Proposition \ref{KerTPn},  $Ker(\T_t)=Ker(MA)\cap Ker(MDMD)$.
 Thanks to our assumption on the dimension of $Im(\delta_1)=Ker(MA)$ and to the fact that the entries of $MA$ are in $\F_p$,
 we have $\dim_{\C_\infty}( Ker(MA)\cap Ker(MDMD))\leqslant 1$ and we can fix a generator   
 $\underline{a}=(a_1,\dots,a_n)\in \F_p^n$. Our goal is to prove $\underline{a}=0$.\\
We prove the even dimension case, for odd $n$ the argument is exactly the same: the vector $\underline{a}$ satisfies the following 
equations coming from $MA\underline{a}=0$:
\begin{equation}\label{v1} \left\{\begin{array}{l}
(m_{1,1}-1)a_1+m_{1,2}a_2+\cdots + m_{1,\frac{n}{2}}a_{\frac{n}{2}}+(-1)^{j+1}m_{1,\frac{n}{2}}a_{\frac{n}{2}+1}+\cdots + (-1)^{j+1}m_{1,1}a_n=0\\
m_{2,1}a_1+(m_{2,2}-1)a_2+\cdots + m_{2,\frac{n}{2}}a_{\frac{n}{2}}+(-1)^{j+1}m_{2,\frac{n}{2}}a_{\frac{n}{2}+1}+\cdots + (-1)^{j+1}m_{2,1}a_n=0\\
\vdots \\
m_{\frac{n}{2},1}a_1+m_{\frac{n}{2},2}a_2+\cdots + (m_{\frac{n}{2},\frac{n}{2}}-1)a_{\frac{n}{2}}+(-1)^{j+1}m_{\frac{n}{2},\frac{n}{2}}a_{\frac{n}{2}+1}+\cdots + (-1)^{j+1}m_{\frac{n}{2},1}a_n=0\\
m_{\frac{n}{2}+1,1}a_1+m_{\frac{n}{2}+1,2}a_2+\cdots +m_{\frac{n}{2}+1,\frac{n}{2}-1}a_{\frac{n}{2}-1}-a_{\frac{n}{2}+1}+\cdots +(-1)^{j+1}m_{\frac{n}{2}+1,1}a_n=0\\
\vdots \\
m_{n-1,1}a_1-a_{n-1}+(-1)^{j+1}m_{n-1,1}a_n=0\\
a_n=0
\end{array}
\right. \,.
\end{equation}

\noindent Now put $\underline{p(t)}:=MD\underline{a}\in \F_p[t]^n$, then (with $a_n=0$)
\begin{small}
\begin{align}\label{p1}
 \underline{p(t)}\! = \! \left(\!\!\!\begin{array}{c}
 p_1(t)\\
p_2(t)\\
\vdots\\
p_{\frac{n}{2}}(t)\\
p_{\frac{n}{2}+1}(t)\\
\vdots\\
p_{n-1}(t)\\
p_n(t)
\end{array}\!\!\right) \! =  \!  \left(\!\!\! \begin{array}{c}
m_{1,1} a_1t^{s_1}+  \cdots + m_{1,\frac{n}{2}}a_{\frac{n}{2}}t^{s_{\frac{n}{2}}}+  (-1)^{j+1}m_{1,\frac{n}{2}}a_{\frac{n}{2}+1}t^{s_{\frac{n}{2}+1}}+
  \cdots  + (-1)^{j+1}m_{1,2}a_{n-1}t^{s_{n-1}}\\
m_{2,1} a_1t^{s_1}+ \cdots +  m_{2,\frac{n}{2}} a_{\frac{n}{2}}t^{s_{\frac{n}{2}}}+  (-1)^{j+1}m_{2,\frac{n}{2}}a_{\frac{n}{2}+1}t^{s_{\frac{n}{2}+1}}+   \cdots   +(-1)^{j+1}(m_{2,2}-1)a_{n-1}t^{s_{n-1}}\\
\vdots \\
m_{\frac{n}{2},1} a_1t^{s_1}+\cdots +  m_{\frac{n}{2},\frac{n}{2}}a_{\frac{n}{2}}t^{s_{\frac{n}{2}}}+   (-1)^{j+1}(m_{\frac{n}{2},\frac{n}{2}}-1)a_{\frac{n}{2}+1}t^{s_{\frac{n}{2}}+1}+   \cdots  +
(-1)^{j+1}m_{\frac{n}{2},2}a_{n-1}t^{s_{n-1}}\\
m_{\frac{n}{2}+1,1} a_1t^{s_1}+\cdots   +(-1)^ja_{\frac{n}{2}}t^{s_{\frac{n}{2}}}+     m_{\frac{n}{2}+1,\frac{n}{2}-1}a_{\frac{n}{2}+2}t^{s_{\frac{n}{2}+2}} + \cdots   +(-1)^{j+1}m_{\frac{n}{2}+1,2}a_{n-1}t^{s_{n-1}}\\
\vdots \\
m_{n-1,1} a_1t^{s_1}+ (-1)^j a_2t^{s_2}\\
(-1)^j a_1t^{s_1}
\end{array}\!\!\!\right) .
\end{align}
\end{small}

\noindent Since $MD\underline{p(t)}=0$, we also have equations:
%\begin{small}
\begin{equation}\label{w1}
\left\{\begin{array}{l}
m_{1,1}t^{s_1}p_1(t)+  \cdots + m_{1,\frac{n}{2}}t^{s_{\frac{n}{2}}}p_{\frac{n}{2}}(t)+  (-1)^{j+1}m_{1,\frac{n}{2}}t^{s_{\frac{n}{2}+1}}p_{\frac{n}{2}+1}(t)+
  \cdots  +   (-1)^{j+1}(m_{1,1}-1)t^{s_n}p_n(t)=0\\
m_{2,1}t^{s_1}p_1(t)+ \cdots +  m_{2,\frac{n}{2}} t^{s_{\frac{n}{2}}}p_{\frac{n}{2}}(t)+  (-1)^{j+1}m_{2,\frac{n}{2}}t^{s_{\frac{n}{2}+1}}p_{\frac{n}{2}+1}(t)+   \cdots  +   (-1)^{j+1}m_{2,1}t^{s_n}p_n(t)=0\\
\vdots \\
m_{\frac{n}{2},1}t^{s_1}p_1(t)+\cdots +  m_{\frac{n}{2},\frac{n}{2}}t^{s_{\frac{n}{2}}}p_{\frac{n}{2}}(t)+   (-1)^{j+1}(m_{\frac{n}{2},\frac{n}{2}}-1)t^{s_{\frac{n}{2}+1}}p_{\frac{n}{2}+1}(t)+   \cdots  +
 (-1)^{j+1}m_{\frac{n}{2},1}t^{s_n}p_n(t)=0\\
m_{\frac{n}{2}+1,1}t^{s_1}p_1(t)+\cdots   +(-1)^jt^{s_{\frac{n}{2}}}p_{\frac{n}{2}}(t)+(-1)^{j+1}m_{\frac{n}{2}+1,\frac{n}{2}-1}t^{s_{\frac{n}{2}+2}}p_{\frac{n}{2}+2}(t)+     \cdots   +
  (-1)^{j+1}m_{\frac{n}{2}+1,1}t^{s_n}p_n(t)=0\\
\vdots \\
m_{n-1,1}t^{s_1}p_1(t)+ (-1)^j t^{s_2}p_2(t)+ (-1)^{j+1}m_{n-1,1}t^{s_n}p_n(t)=0\\
(-1)^jt^{s_1}p_1(t)=0
\end{array} \right.\,.
\end{equation}
%\end{small}

\noindent Note that in \eqref{w1} we have polynomials in $\F_p[t]$, from now on we shall basically use the identity principle
for polynomials to solve the equations in the $a_i$. From the last row in \eqref{w1} we get $p_1(t)=0$, i.e. comparing with
\eqref{p1} 
\[ m_{1,1} a_1= m_{1,2}a_2= \cdots = m_{1,\frac{n}{2}}a_{\frac{n}{2}}= m_{1,\frac{n}{2}}a_{\frac{n}{2}+1}=
  \cdots =   m_{1,2}a_{n-1}=0\,. \]
Substituting in the first and second-last equations in \eqref{v1} we obtain
  $$a_1=a_{n-1}=0 $$
  which also means that $p_n(t)=0$.\\
We can rewrite \eqref{v1}, \eqref{p1} and \eqref{w1} as

\begin{equation}\label{v2} \left\{\begin{array}{l}
(m_{2,2}-1)a_2+\cdots + m_{2,\frac{n}{2}}a_{\frac{n}{2}}+(-1)^{j+1}m_{2,\frac{n}{2}}a_{\frac{n}{2}+1}+\cdots + (-1)^{j+1}m_{2,3}a_{n-2}=0\\
\vdots \\
m_{\frac{n}{2},2}a_2+\cdots + (m_{\frac{n}{2},\frac{n}{2}}-1)a_{\frac{n}{2}}+(-1)^{j+1}m_{\frac{n}{2},\frac{n}{2}}a_{\frac{n}{2}+1}+\cdots + (-1)^{j+1}m_{\frac{n}{2},3}a_{n-2}=0\\
m_{\frac{n}{2}+1,2}a_2+\cdots+m_{\frac{n}{2}+1,\frac{n}{2}-1}a_{\frac{n}{2}-1}-a_{\frac{n}{2}+1}+\cdots +(-1)^{j+1}m_{\frac{n}{2}+1,3}a_{n-2}=0\\
\vdots \\
m_{n-2,2}a_2-a_{n-2}=0\\
a_1=a_{n-1}=a_n=0
\end{array}
\right. ,
\end{equation}

\begin{small}
\begin{equation}\label{p2}
\left(\begin{array}{c}
 p_1(t)\\
p_2(t)\\
\vdots\\
p_{\frac{n}{2}}(t)\\
p_{\frac{n}{2}+1}(t)\\
\vdots\\
p_{n-1}(t)\\
p_n(t)
\end{array}\right) =\left( \begin{array}{c}
0\\
m_{2,2}a_2t^{s_2}+ \cdots +  m_{2,\frac{n}{2}} a_{\frac{n}{2}}t^{s_{\frac{n}{2}}}+  (-1)^{j+1}m_{2,\frac{n}{2}}a_{\frac{n}{2}+1}t^{s_{\frac{n}{2}+1}}+   \cdots   +(-1)^{j+1}m_{2,3}a_{n-2}t^{s_{n-2}}\\
\vdots \\
m_{\frac{n}{2},2} a_2t^{s_2}+\cdots +  m_{\frac{n}{2},\frac{n}{2}}a_{\frac{n}{2}}t^{s_{\frac{n}{2}}}+   (-1)^{j+1}(m_{\frac{n}{2},\frac{n}{2}}-1)a_{\frac{n}{2}+1}t^{s_{\frac{n}{2}}+1}+   \cdots  +
(-1)^{j+1}m_{\frac{n}{2},3}a_{n-2}t^{s_{n-2}}\\
m_{\frac{n}{2}+1,2} a_2t^{s_2}+\cdots   +(-1)^ja_{\frac{n}{2}}t^{s_{\frac{n}{2}}}+     m_{\frac{n}{2}+1,\frac{n}{2}-1}a_{\frac{n}{2}+2}t^{s_{\frac{n}{2}+2}} + \cdots   +(-1)^{j+1}m_{\frac{n}{2}+1,3}a_{n-2}t^{s_{n-2}}\\
\vdots \\
 (-1)^j a_2t^{s_2}\\
0
\end{array}\right)
\end{equation}
\end{small}
and
\begin{equation}\label{w2}\!\!\!
\left\{\begin{array}{l}
m_{1,2}t^{s_2}p_2(t)+  \cdots + m_{1,\frac{n}{2}}t^{s_{\frac{n}{2}}}p_{\frac{n}{2}}(t)+  (-1)^{j+1}m_{1,\frac{n}{2}}t^{s_{\frac{n}{2}+1}}p_{\frac{n}{2}+1}(t)+
  \cdots     (-1)^{j+1}m_{1,2}t^{s_{n-1}}p_{n-1}(t)=0\\
 m_{2,2}t^{s_2}p_2(t)+ \cdots +  m_{2,\frac{n}{2}} t^{s_{\frac{n}{2}}}p_{\frac{n}{2}}(t)+  (-1)^{j+1}m_{2,\frac{n}{2}}t^{s_{\frac{n}{2}+1}}p_{\frac{n}{2}+1}(t)+   \cdots  +   (-1)^{j+1}(m_{2,2}-1)t^{s_{n-1}}p_{n-1}(t)=0\\
\vdots \\
m_{\frac{n}{2},2} t^{s_2}p_2(t)+\cdots +  m_{\frac{n}{2},\frac{n}{2}}t^{s_{\frac{n}{2}}}p_{\frac{n}{2}}(t)+   (-1)^{j+1}(m_{\frac{n}{2},\frac{n}{2}}-1)t^{s_{\frac{n}{2}+1}}p_{\frac{n}{2}+1}(t)+   \cdots  +
 (-1)^{j+1}m_{\frac{n}{2},2}t^{s_{n-1}}p_{n-1}(t)=0\\
m_{\frac{n}{2}+1,2} t^{s_2}p_2(t)+\cdots   +(-1)^jt^{s_{\frac{n}{2}}}p_{\frac{n}{2}}(t)+(-1)^{j+1}m_{\frac{n}{2}+1,\frac{n}{2}-1}t^{s_{\frac{n}{2}+2}}p_{\frac{n}{2}+2}(t)+     \cdots   +
  (-1)^{j+1}m_{\frac{n}{2}+1,2}t^{s_{n-1}}p_{n-1}(t)=0\\
\vdots \\
(-1)^j t^{s_2}p_2(t)=0\\
p_1(t)=p_n(t)=0
\end{array} \right. \!\!\!.
\end{equation}

\noindent We repeat the same argument starting now from the second-last equation in \eqref{w2}, which yields $p_2(t)=0$. This means
\[  m_{2,2}a_2= \cdots =  m_{2,\frac{n}{2}} a_{\frac{n}{2}}=  m_{2,\frac{n}{2}}a_{\frac{n}{2}+1}=   \cdots   =m_{2,3}a_{n-2}=0,\]
which, substituted in the first equation of \eqref{v2}, gives $a_2=0$. Thus (second-last equations in \eqref{v2} and \eqref{p2}) 
$a_{n-2}=0$ and $p_{n-1}(t)=0$ as well.\\
Iterating the process we see that the specular symmetries between $MD$ ($(-1)^j$ on the antidiagonal) and $MA$ ($-1$ on the diagonal)  
lead to $\underline{a}=0$.
\end{proof}

\subsection{Direct sum}

\begin{thm}\label{ThmDirSum}
Assume that $\dim_{\C_\infty}Im (\delta_1)=1$.  Then $S^1_{k,m}(\G_0(t))=S^{1,old}_{k,m}(\G_0(t))\oplus S^{1,new}_{k,m}(\G_0(t))$.
\end{thm}

\begin{proof}
We use the criterion of Theorem \ref{ThmCriterionDirSum}, where we noted that an element $\eta\in Ker( \D)$ must be in $Im(\delta_1)=Ker(MA)$ 
as well, and that $\U_t(\eta)$ is both old and new. We take $\underline{a}\in \F_p^n$ which verifies $MA\underline{a}=0$ and 
represents an element $\eta=\delta_1\varphi\in Ker(\D)$, then $Tr'(\U_t(\delta_1\varphi))=0$, i.e. $TF(MD\underline{a})=0$ and we
prove that these two relations yield $\underline{a}=0$, so that $Ker(\D)=0$ and $\D$ is invertible. 
As before we only treat the case of even $n$.

\noindent The equation $MA\underline{a}=0$ gives again the system \eqref{v1} (in particular $a_n=0$), then, writing 
$\underline{p(t)}=MD\underline{a}$ as in \eqref{p1}, from $TF(MD\underline{a})=0$ we get
\begin{align}\label{d1}
\left\{ \begin{array}{l}
m_{1,1}t^{s_1}p_1(t)+  \cdots + m_{1,\frac{n}{2}}t^{s_{\frac{n}{2}}}p_{\frac{n}{2}}(t)+  m_{1,\frac{n}{2}}(-t)^{s_{\frac{n}{2}+1}}p_{\frac{n}{2}+1}(t)+
  \cdots  +  m_{1,1}(-t)^{s_n}p_n(t)=0\\
m_{2,1}t^{s_1}p_1(t)+ \cdots +  m_{2,\frac{n}{2}} t^{s_{\frac{n}{2}}}p_{\frac{n}{2}}(t)+ m_{2,\frac{n}{2}}(-t)^{s_{\frac{n}{2}+1}}p_{\frac{n}{2}+1}(t)+   \cdots  +   m_{2,1}(-t)^{s_n}p_n(t)=0\\
\vdots \\
m_{\frac{n}{2},1}t^{s_1}p_1(t)+\cdots +  m_{\frac{n}{2},\frac{n}{2}}t^{s_{\frac{n}{2}}}p_{\frac{n}{2}}(t)+   m_{\frac{n}{2},\frac{n}{2}}(-t)^{s_{\frac{n}{2}+1}}p_{\frac{n}{2}+1}(t)+   \cdots  +
 m_{\frac{n}{2},1}(-t)^{s_n}p_n(t)=0\\
m_{\frac{n}{2}+1,1}t^{s_1}p_1(t)+\cdots   +m_{\frac{n}{2}+1,1}(-t)^{s_n}p_n(t)=0\\
\vdots \\
m_{n-2,1}t^{s_1}p_1(t)+ m_{n-2,2}t^{s_2}p_2(t)+ m_{n-2,2}(-t)^{s_{n-1}}p_{n-1}(t)+m_{n-2,1}(-t)^{s_n}p_n(t)=0\\
m_{n-1,1}t^{s_1}p_1(t)+ m_{n-1,1}(-t)^{s_n}p_n(t)=0
\end{array}
\right.
\end{align}
(the matrix $T$ can be taken from Remark \ref{RemMatrixT}). 

\noindent In the last equation of \eqref{d1} the term with the highest degree in $t$ is $m_{n-1,1}(-t)^{s_n}(-1)^ja_1t^{s_1}=
-m_{n-1,1}a_1t^k$ (note that $p_1(t)$ has degree at most $s_{n-1}$ because $a_n=0$): therefore $m_{n-1,1}a_1=0$ and the
second-last equation in \eqref{v1} tell us that $a_{n-1}=0$. 
Now \eqref{v1} and \eqref{p1} turn into

\begin{align}\label{s2}
\left\{ \begin{array}{l}
(m_{1,1}-1)a_1+m_{1,2}a_2+\cdots + m_{1,\frac{n}{2}}a_{\frac{n}{2}}+(-1)^{j+1}m_{1,\frac{n}{2}}a_{\frac{n}{2}+1}+\cdots + (-1)^{j+1}m_{1,3}a_{n-2}=0\\
m_{2,1}a_1+(m_{2,2}-1)a_2+\cdots + m_{2,\frac{n}{2}}a_{\frac{n}{2}}+(-1)^{j+1}m_{2,\frac{n}{2}}a_{\frac{n}{2}+1}+\cdots + (-1)^{j+1}m_{2,3}a_{n-2}=0\\
\vdots \\
m_{\frac{n}{2},1}a_1+m_{\frac{n}{2},2}a_2+\cdots + (m_{\frac{n}{2},\frac{n}{2}}-1)a_{\frac{n}{2}}+(-1)^{j+1}m_{\frac{n}{2},\frac{n}{2}}a_{\frac{n}{2}+1}+\cdots + (-1)^{j+1}m_{\frac{n}{2},3}a_{n-2}=0\\
m_{\frac{n}{2}+1,1}a_1+m_{\frac{n}{2}+1,2}a_2+\cdots +m_{\frac{n}{2}+1,\frac{n}{2}-1}a_{\frac{n}{2}-1}-a_{\frac{n}{2}+1}+\cdots +(-1)^{j+1}m_{\frac{n}{2}+1,3}a_{n-2}=0\\
\vdots \\
m_{n-2,1}a_1+m_{n-2,2}a_2-a_{n-2}=0\\
m_{n-1,1}a_1=0\\
a_{n-1}=a_n=0
\end{array}\right.
\end{align}

\begin{small}
\begin{align}\label{md2}
 \underline{p(t)}\! = \! \left(\!\!\!\begin{array}{c}
 p_1(t)\\
p_2(t)\\
\vdots\\
p_{\frac{n}{2}}(t)\\
p_{\frac{n}{2}+1}(t)\\
\vdots\\
p_{n-1}(t)\\
p_n(t)
\end{array}\!\!\right) \! =  \!  \left(\!\!\! \begin{array}{c}
m_{1,1} a_1t^{s_1}+  \cdots + m_{1,\frac{n}{2}}a_{\frac{n}{2}}t^{s_{\frac{n}{2}}}+  (-1)^{j+1}m_{1,\frac{n}{2}}a_{\frac{n}{2}+1}t^{s_{\frac{n}{2}+1}}+
  \cdots  + (-1)^{j+1}m_{1,3}a_{n-2}t^{s_{n-2}}\\
m_{2,1} a_1t^{s_1}+ \cdots +  m_{2,\frac{n}{2}} a_{\frac{n}{2}}t^{s_{\frac{n}{2}}}+  (-1)^{j+1}m_{2,\frac{n}{2}}a_{\frac{n}{2}+1}t^{s_{\frac{n}{2}+1}}+   \cdots   +(-1)^{j+1}m_{2,3}a_{n-2}t^{s_{n-2}}\\
\vdots \\
m_{\frac{n}{2},1} a_1t^{s_1}+\cdots +  m_{\frac{n}{2},\frac{n}{2}}a_{\frac{n}{2}}t^{s_{\frac{n}{2}}}+   (-1)^{j+1}(m_{\frac{n}{2},\frac{n}{2}}-1)a_{\frac{n}{2}+1}t^{s_{\frac{n}{2}}+1}+   \cdots  +
(-1)^{j+1}m_{\frac{n}{2},3}a_{n-2}t^{s_{n-2}}\\
m_{\frac{n}{2}+1,1} a_1t^{s_1}+\cdots   +(-1)^ja_{\frac{n}{2}}t^{s_{\frac{n}{2}}}+     m_{\frac{n}{2}+1,\frac{n}{2}-1}a_{\frac{n}{2}+2}t^{s_{\frac{n}{2}+2}} + \cdots   +(-1)^{j+1}m_{\frac{n}{2}+1,3}a_{n-2}t^{s_{n-2}}\\
\vdots \\
 (-1)^j a_2t^{s_2}\\
(-1)^j a_1t^{s_1}
\end{array}\!\!\!\right)\,.
\end{align}
\end{small}

\noindent Now consider the second-last equation in \eqref{d1} 
$$m_{n-2,1}t^{s_1}p_1(t)+m_{n-2,2}t^{s_2}p_2(t)+m_{n-2,2}(-t)^{s_{n-1}}p_{n-1}(t)+m_{n-2,1}(-t)^{s_n}p_n(t)=0.$$
The term with the highest possible degree $s_1+s_n=s_2+s_{n-1}=k$ is 
\[ m_{n-2,2}(-t)^{s_{n-1}}(-1)^ja_2t^{s_2}+m_{n-2,1}(-t)^{s_n}(-1)^ja_1t^{s_1}=
-(m_{n-2,2}a_2+m_{n-2,1}a_1)t^k \,,\] 
hence $m_{n-2,1}a_1+m_{n-2,2}a_2=0$.  Looking at the system \eqref{s2} we obtain $a_{n-2}=0$ and
$p_{n-3}(t)=m_{n-3,1}a_1t^{s_1}+\cdots+(-1)^ja_4t^{s_4}$ (it is relevant that it has degree
at most $s_4=s_{n+1-(n-3)}$, i.e. that the terms of higher degree vanish).

\noindent The proof goes on in the same way: it may be less evident than the one of Theorem \ref{ThmInj} (where the $a_i$ vanished in
couples), but looking always at the terms of degree $k$ of the $(n-i)$-th equation of \eqref{d1} we are able to prove
that $a_{n-i}=0$ and, as an immediate consequence from \eqref{p1}, that $p_{n-i-1}(t)$ has degree at most $s_{i+2}$.  
For example midway through the proof we get
\[  \underline{a}=\left(\begin{array}{c}  a_1\\ \vdots \\ a_{\frac{n}{2}}\\ 0\\ \vdots \\ 0 \end{array} \right)\quad 
\mathrm{and} \quad \underline{p(t)}=\left(\begin{array}{c} m_{1,1} a_1t^{s_1}+  \cdots + m_{1,\frac{n}{2}}a_{\frac{n}{2}}t^{s_{\frac{n}{2}}}\\
\vdots\\
m_{\frac{n}{2},1} a_1t^{s_1}+\cdots +  m_{\frac{n}{2},\frac{n}{2}}a_{\frac{n}{2}}t^{s_{\frac{n}{2}}}\\
m_{\frac{n}{2}+1,1} a_1t^{s_1}+\cdots + (-1)^ja_{\frac{n}{2}}t^{s_{\frac{n}{2}}}\\
m_{\frac{n}{2}+2,1} a_1t^{s_1}+\cdots + (-1)^ja_{\frac{n}{2}-1}t^{s_{\frac{n}{2}-1}}\\
\vdots\\
(-1)^ja_1t^{s_1}
 \end{array} \right)\,.  \]
Therefore, what remains of \eqref{v1} is
\begin{equation}\label{s3}
\left\{  
\begin{array}{l}
(m_{1,1}-1)a_1+m_{1,2}a_2+\cdots + m_{1,\frac{n}{2}}a_{\frac{n}{2}}=0\\
m_{2,1}a_1+(m_{2,2}-1)a_2+\cdots + m_{2,\frac{n}{2}}a_{\frac{n}{2}}=0\\
\vdots \\
m_{\frac{n}{2},1}a_1+m_{\frac{n}{2},2}a_2+\cdots + (m_{\frac{n}{2},\frac{n}{2}}-1)a_{\frac{n}{2}}=0\\
a_{\frac{n}{2}+1}=\cdots =a_n=0
\end{array}\,.
\right.
\end{equation}
Finally, we observe that the $\frac{n}{2}$-th equation of \eqref{d1} is
\[ m_{\frac{n}{2},1}t^{s_1}p_1(t)+\cdots+ m_{\frac{n}{2},\frac{n}{2}}t^{s_{\frac{n}{2}}}p_{\frac{n}{2}}(t)+m_{\frac{n}{2},\frac{n}{2}}(-t)^{s_{\frac{n}{2}+1}}p_{\frac{n}{2}+1}(t)+\cdots+m_{\frac{n}{2},1}(-t)^{s_n}p_n(t)=0 \,. \]
As before, the term of degree $k$ must have coefficient 0 and it appears only from 
$m_{\frac{n}{2},\frac{n}{2}}(-t)^{s_{\frac{n}{2}+1}}p_{\frac{n}{2}+1}(t)$ on, so we get
\[ m_{\frac{n}{2},\frac{n}{2}}a_{\frac{n}{2}}+m_{\frac{n}{2},\frac{n}{2}-1}a_{\frac{n}{2}-1}+ \cdots + m_{\frac{n}{2},1}a_1=0  \]
and, by \eqref{s3}, $a_\frac{n}{2}=0$ as well. \\
Iterating we get $\underline{a}=0$ and so our claim.
\end{proof}

\end{document}